\documentclass[12pt,a4paper]{amsart}
\usepackage[utf8]{inputenc}
\usepackage[hmargin=30mm, vmargin=25mm, includefoot, twoside]{geometry}
\usepackage[T1]{fontenc}
\usepackage{amsmath}
\usepackage{amsfonts}
\usepackage{amssymb}
\usepackage{xcolor}
\definecolor{darkred}{RGB}{180,0,0}
\definecolor{darkblue}{RGB}{0,0,180}
\usepackage{hyperref}
\hypersetup{
  unicode = true,
  colorlinks = true,
  urlcolor   = darkblue,
  linkcolor  = darkblue,
  citecolor  = darkred,
  pdftitle = {Two conjectures on coarse conjugacy by Geller and Misiurewicz},
  pdfauthor = {Damian Sawicki},
  pdfkeywords = {coarse conjugacy, coarse equivalence, close maps, coarse inverse, coarse dynamics, coarse entropy}
}
\usepackage[initials]{amsrefs}
\usepackage[capitalise]{cleveref}

\usepackage[normalem]{ulem}
\usepackage{txfonts, pxfonts}
\usepackage[super]{nth}

\title[Two conjectures on coarse conjugacy]
{Two conjectures
on coarse conjugacy
\\
by Geller and Misiurewicz}
\author{Damian Sawicki}
\address{KU Leuven, Department of Mathematics,
Celestijnenlaan 200b – box 2400, 3001 Leuven, Belgium}
\urladdr{https://sites.google.com/view/damiansawicki}
\thanks{The author was partially supported by the FWO research project G090420N of the Research Foundation Flanders.}
\date{\nth{28} June 2022}
\subjclass[2020]{37B02, 51F30}
\keywords{Coarse conjugacy, coarse equivalence, close maps, coarse inverse, coarse dynamics, coarse entropy}

\newcommand{\st}[1][ ]{\ #1|\ }
\newcommand{\ceil}[1]{{\left\lceil #1 \right\rceil}}
\newcommand{\floor}[1]{{\left\lfloor #1 \right\rfloor}}

\DeclareMathOperator{\id}{id}
\DeclareMathOperator{\N}{N}

\newcommand\set[1]{\left\{#1\right\}}
\newcommand\range[2]{#1,\ldots,#2}
\newcommand*{\defeq}{\mathrel{\vcenter{\baselineskip0.5ex \lineskiplimit0pt
                     \hbox{\scriptsize.}\hbox{\scriptsize.}}}%
                     =}
\newcommand*{\eqdef}{=\mathrel{\vcenter{\baselineskip0.5ex \lineskiplimit0pt
                     \hbox{\scriptsize.}\hbox{\scriptsize.}}}}
\def\NN{{\mathbb N}}
\def\ZZ{{\mathbb Z}}
\def\RR{{\mathbb R}}

\newtheorem{thm}{Theorem}[section]
\crefname{thm}{Theorem}{Theorems}

\crefname{conj}{Conjecture}{Conjectures}
\newtheorem{lem}[thm]{Lemma}
\newtheorem{rem}[thm]{Remark}
\newtheorem{prop}[thm]{Proposition}
\newtheorem{cor}[thm]{Corollary}
\newtheorem{conjecture}{Conjecture}

\crefname{conjecture}{Conjecture}{Conjectures}
\newtheorem{alphaTheorem}[conjecture]{Theorem}

\crefname{alphaTheorem}{Theorem}{Theorems}

\theoremstyle{definition}
\newtheorem{defi}[thm]{Definition}
\newtheorem{longEx}[thm]{Example}
\crefname{longEx}{Example}{Examples}

\begin{document}

\begin{abstract}
In their study of coarse entropy, W.~Geller and M.~Misiurewicz introduced the notion of coarse conjugacy: a version of conjugacy appropriate for dynamics on metric spaces observed from afar. They made two conjectures on coarse conjugacy generalising their results. We disprove both of these conjectures. We investigate the impact of extra assumptions on the validity of the conjectures: We show that the result of Geller and Misiurewicz towards one of the conjectures can be considered optimal, and we prove the other conjecture under an assumption complementary to that from the referenced work. 
\end{abstract}

\maketitle

\section{Introduction}

The following notion of conjugacy was recently introduced by W.~Geller and M.~Misiurewicz as a fundamental concept in their study \cite{GM} of coarse entropy, which is a version of entropy reminiscent of topological entropy and consistent with the coarse-geometric approach to metric spaces.

\begin{defi}\label{conjugacy}
Self-maps $f\colon X\to X$ and $g\colon Y\to Y$ of metric spaces $X$ and $Y$ are \emph{coarsely conjugate} if there
exists a coarse equivalence $\varphi\colon X\to Y$ with a coarse inverse $\psi \colon Y\to X$ such that
$\varphi\circ f$ is close to $g\circ\varphi$ and $\psi\circ g$ is close to
$f\circ\psi$.
\end{defi}

Let us recall the definitions. Assume that $(X,d_X)$ and $(Y,d_Y)$ are metric spaces. Two maps $a,b\colon X\to Y$ are \emph{close} if their supremum distance is finite, that is, $\sup_{x\in X} d_Y(a(x), b(x))<\infty$. A~map $\varphi\colon X\to Y$ is \emph{controlled} (or \emph{bornologous}) if there exists a non-decreasing function $\rho\colon [0,\infty) \to [0,\infty)$ controlling how much $\varphi$ increases distances: Namely, every $x,x'\in X$ satisfy $d_Y(\varphi(x), \varphi(x')) \leq \rho(d_X(x,x'))$. A~controlled map $\varphi \colon X\to Y$ is called a \emph{coarse equivalence} if it admits a \emph{coarse inverse}, that is, a controlled map $\psi\colon Y\to X$ such that $\varphi \circ \psi$ and $\psi\circ \varphi$ are close to the respective identity maps. 
For a comprehensive introduction to coarse geometry, see \cites{Roe, NY}.

Inspired by Gromov's classical notion of coarse equivalence, coarse conjugacy is designed to capture the fact that dynamics on two unbounded metric spaces is the same at the large scale.

\subsubsection*{Context} 
Coarse geometry or large scale geometry is a theory that is dual to topology of metric spaces in the sense that it focuses on large rather than small distances. A prominent class of discrete metric spaces whose coarse geometry is rich is the class of finitely generated groups, including fundamental groups of closed manifolds. For example, coarse geometry provided the proof \cite{Yu} of the Novikov conjecture for a very large class of manifolds.

Compact dynamical systems yield unbounded metric spaces via the warped cone construction \cite{Roe-cones}, and under certain assumptions \cite{FNvL} warped cones associated to two dynamical systems are coarsely equivalent if and only if the initial dynamical systems are conjugate.

Once a class of mathematical objects has been well understood ``statically'', the researchers aspire to unravel their dynamics. Topological dynamics is a~classical example, and more modern ones include C*-dy\-nam\-ics or the study of automorphism groups of various structures.

The considerations of coarse-geometric notions of dynamics
are therefore timely. Indeed, while coarse entropy of \cite{GM} is inspired by topological entropy, another notion referred to as coarse entropy and inspired by algebraic entropy of group endomorphisms has also been introduced recently \cite{Zava}.
The introduction of coarse conjugacy is a fundamental development
within this emerging direction,
and our aim in the present note is to analyse two conjectures on coarse conjugacy by Geller and Misiurewicz.

\subsubsection*{Results}

The following theorem from \cite{GM} shows that coarse conjugacy is preserved when replacing transformations with their $n$th iterations.

\begin{thm}\cite{GM}*{Proposition~2.8 and Corollary~2.13}\label{thmPowers}
If $f$ and $g$ are coarsely conjugate and $g$ is controlled, then $f$ is controlled as well and for any natural $n$ the maps $f^n$ and $g^n$ are coarsely conjugate via the same
coarse equivalences as $f$ and~$g$.
\end{thm}

Geller and Misiurewicz made the following conjecture generalising the above result.

\begin{conjecture}\cite{GM}*{Conjecture~2.11}\label{conjecture2}
If $f$ and $g$ are coarsely conjugate, then so are $f^n$ and $g^n$ for all natural $n$.
\end{conjecture}

The analogues of \cref{conjecture2} hold not only in the category of sets or topological spaces, but also in more complicated settings like measurable dynamics. However, our \cref{thm:squares,invertible} show that
\cref{conjecture2} fails.
The corresponding counterexamples are \cref{ex:squares,ex:invertible}.

In fact, in \cref{ex:squares,ex:invertible} we arrange $X$ to equal $Y$ and the conjugating maps $\varphi$ and $\psi$ to be the identity map. This shows that imposing any extra conditions on the conjugating maps, for example bijectivity like in \cite{Zava}, does not make \cref{conjecture2} true. While this is not yet true in \cref{ex:squares}, we construct the transformations $f$ and $g$ in \cref{ex:invertible} to be bijections. Summarising, we prove that \cref{conjecture2} does not become true if one assumes the conjugating maps to be the identity and the dynamics to be invertible, which shows the optimality of \cref{thmPowers}.

One can ask how long it takes in terms of $n$ in \cref{conjecture2} to find the powers $f^n$ and $g^n$ that are not coarsely conjugate. It turns out that the answer can be any integer $\geq 2$. For every $k\in \NN_{\geq 1}$, each of \cref{ex:squares,ex:invertible} constructs $f\colon X\to X$ and $g\colon Y\to Y$ such that $f^n$ and~$g^n$ are coarsely conjugate precisely for $n\in \set{\range{1}{k}}$.

\bigskip

The following result gives a sufficient condition for coarse conjugacy.

\begin{thm}\cite{GM}*{Proposition~2.8}\label{thmEmbeddings}
Consider maps $f\colon X\to X$ and $g\colon Y\to Y$ for which there exists a coarse equivalence $\varphi \colon X\to Y$ such that $\varphi\circ f$ is close to $g\circ\varphi$ and $g$ is controlled. Then $f$ is also controlled and for any coarse inverse $\psi$ of $\varphi$ the maps $f$ and $g$ are coarsely conjugate via $\varphi$ and $\psi$.
\end{thm}

Geller and Misiurewicz conjectured a generalisation as follows.

\begin{conjecture}\cite{GM}*{Conjecture~2.4}\label{conjecture}
If there exist coarse equivalences $\varphi\colon X\to Y$ and $\psi\colon Y\to X$ such that
$\varphi\circ f$ is close to $g\circ\varphi$ and $\psi\circ g$ is close to
$f\circ\psi$, then $f$ and $g$ are coarsely conjugate.
\end{conjecture}

In other words, \cref{conjecture} predicts that if the dynamics on $X$ can be realised inside $Y$, and the dynamics on $Y$ can be realised inside $X$, and these can be realised through maps $\varphi$ and $\psi$ that are coarse equivalences,
then $\varphi$ and $\psi$ can be chosen in a compatible way, that is, being coarse inverses of each other.

However, in \cref{ex:disjointTheorem} and the corresponding \cref{disjointTheorem} we show that \cref{conjecture}
does not hold.
In fact, in our counterexample the transformations $f$ and~$g$ are bijective, which proves that---like \cref{conjecture2}---\cref{conjecture} is not true after
restricting to invertible dynamics.

Nonetheless, as opposed to \cref{conjecture2}, \cref{conjecture} can be turned into a~theorem by imposing extra set-theoretic restrictions on the conjugating maps~$\varphi$ or~$\psi$. Namely, the following \cref{thmSurjective} (see \cref{surjective-conjecture} in text) shows that \cref{conjecture} holds if one assumes $\varphi$~to be surjective. Furthermore, similarly as in \cref{thmEmbeddings}, it suffices to assume only that $\varphi\circ f$ is close to $g\circ\varphi$, and then the existence of $\psi$ with the appropriate properties follows.

\begin{alphaTheorem}\label{thmSurjective}
Assume that there exists a \emph{surjective} coarse equivalence $\varphi\colon X\to Y$ such that
$\varphi\circ f$ is close to $g\circ\varphi$. Then, for any $\psi\colon Y\to X$ such that $\varphi \circ \psi = \id_Y$, the maps $f$ and $g$ are coarsely conjugate via $\varphi$ and $\psi$.
\end{alphaTheorem}

Interestingly, we prove that \cref{thmSurjective} does not have an analogue for injective conjugating maps. Indeed, not even one but both of the maps $\varphi$ and $\psi$ are injective in \cref{ex:disjointTheorem}.

\Cref{thmEmbeddings} and our \cref{thmSurjective} are complementary because when compared with \cref{conjecture}, the additional assumptions in the former concern the dynamics (namely, the map $g\colon Y\to Y$), and the additional assumptions in the latter concern the conjugating maps (namely, $\varphi\colon X\to Y$).

\bigskip
Finally, in \cref{ex:squares,ex:invertible,ex:disjointTheorem} we provide counterexamples 
whose underlying metric spaces are tame from the coarse-geometric perspective: In particular, they have finite asymptotic dimension and polynomial volume growth. Indeed, the metric spaces in \cref{ex:squares,ex:invertible} 
are halflines (or coarsely equivalent to halflines), while the metric space in \cref{ex:disjointTheorem} 
is a subset of $\RR^3$, and actually 
it can also be coarsely embedded in $\RR^2$. It would be interesting to know if \cref{conjecture2,conjecture} admit counterexamples where the asymptotic dimension of the underlying metric spaces vanishes.

\bigskip We present the constructions in the order of increasing difficulty of the corresponding proofs and tools.
\Cref{powers-section,subsec:main-sec} provide counterexamples to \cref{conjecture,conjecture2}, respectively, and \cref{thmSurjective} is accordingly proved in \cref{subsec:theorem}.

\section{Counterexamples to \texorpdfstring{\cref{conjecture2}}{Conjecture \ref{conjecture2}}}\label{powers-section}

In this section, we construct examples of self-maps $f\colon X\to X$ and $g\colon Y \to Y$ of metric spaces such that $f,g$ are coarsely conjugate, but there exists $n > 1$ such that $f^n$ and $g^n$ are not, that is, these are counterexamples to \cref{conjecture2}. More specifically, for any $k \in \NN_{\geq 1}$,  we obtain examples such that $f^n$ and $g^n$ are coarsely conjugate for all $n\leq k$ but not for any $n>k$.

In \cref{ex:squares} only the map $g$ is a bijection, but in \cref{ex:invertible} both dynamical systems are invertible. In both cases,  the conjugating maps $\varphi$ and $\psi$ are the identity on $X=Y$.

Throughout, for a subset $Z$ of a metric space $(X,d)$ and $B\geq 0$, we will denote by $\N_B(Z)$ the set of all points $x\in X$ such that $\inf_{z\in Z} d(x,z) \leq B$. If two maps $a$ and~$b$ from a set $X$ to a metric space $Y$ are close and $B \geq \sup_{x\in X} d_Y(a(x), b(x))$, we will also say that $a$ and~$b$ are \emph{$B$-close}.

\begin{longEx}\label{ex:squares}
Let $X=Y=[1,\infty)$ and let $g\colon Y\to Y$ be given by $g(x) = x^2$. For every $k\in \NN_{\geq 1}$ define $f_k\colon X\to X$ as $f_k(x) = {\ceil{x^{2^k}}}^{2^{-k+1}}$.
\end{longEx}

\begin{thm}\label{thm:squares}
Let $X=Y=[1,\infty)$, $g\colon Y\to Y$, and $f_k\colon X\to X$ for every $k\in \NN_{\geq 1}$ be as in \cref{ex:squares}. Then for every $k\in \NN_{\geq 1}$
\begin{itemize}
\item for every $n \in \set{\range 1 k}$ the maps $f_k^n$ and~$g^n$ are coarsely conjugate via the identity maps;
\item but for every $n>k$, there is no coarse equivalence $\psi\colon Y\to X$ such that $\psi\circ g^{n}$ and $f_k^{n}\circ \psi$ are close; in particular, $f_k^{n}$ and $g^{n}$ are not coarsely conjugate.
\end{itemize}
\end{thm}

\begin{proof}
Let $k\in \ZZ$. Define $X_k\subseteq X$ and $\phi_k\colon X\to X$ by the following formulae:
\begin{equation}\label{phiK}
X_k = \set{m^{2^{-k}} \st[\big] m\in \NN_{\geq 1}}
\quad\text{and}\quad
\phi_k(x) = {\ceil{x^{2^k}}}^{2^{-k}}.
\end{equation}
Note that $\phi_k$ maps $X$ onto $X_k$ and is the identity on $X_k$.
Furthermore, $g(X_k) = X_{k-1} \subseteq X_k$, so in particular $g$ preserves $X_k$.

Fix $k\in \NN_{\geq 1}$ and note that $f_k = g\circ \phi_k$.
By the above observations, $f_k^n=(g\circ \phi_k)^n = g^n\circ \phi_k$ for every $n\geq 1$. Hence, applying the equality $g(X_k) = X_{k-1}$ (valid for every $k\in \ZZ$), we obtain
\begin{equation}\label{imageF}
f_k^n(X) = g^n \circ \phi_k(X) = g^n(X_k) = X_{k-n}.
\end{equation}

We claim that for $n\leq k$ the maps $f_k^n$ and $g^n$ are close and hence coarsely conjugate via the identity maps. Indeed, since $\phi_k(x)\geq x$ for all $x\in X$, we get that $f_k^n(x) = g^n\circ \phi_k(x) \geq g^n\circ \id(x) = g^n(x)$ for every $k$, so we get $0\leq f_k^n(x) - g^n(x)$. For~$n\leq k$, we have a matching upper bound:
\begin{equation}\label{closeness-ceil}
\begin{aligned}
f_k^n(x) - g^n(x) &= g^n\circ \phi_k(x) - g^n(x) \\
&= \left({\ceil{x^{2^k}}}^{2^{-k}}\right)^{2^n} - x^{2^n} \\
&= {\ceil{x^{2^k}}}^{2^{-k+n}} - \left(x^{2^k}\right)^{2^{-k+n}} \leq 2^{-k+n}, 
\end{aligned}
\end{equation}
where the inequality follows from the mean value theorem because
for $n \leq k$ the derivative of the map $y\mapsto y^{2^{-k+n}}$ is bounded by $2^{-k+n}$ on $[1,\infty)$.

Now, fix $n>k$ and suppose for a contradiction that there exists a coarse equivalence $\psi\colon Y\to X$ such that $\psi\circ g^{n}$ and $f_k^{n}\circ \psi$ are $R$-close for some $R > 0$. Since~$\psi$ is a coarse equivalence, there exists a constant $B > 0$ such that $\N_B(\psi(Y)) = X$ (this standard property of coarse equivalences follows from the fact that $\psi\circ \varphi$ is close to the identity on~$X$ for some coarse inverse~$\varphi$). Clearly, $g^{n}(Y) = Y$ for any $n$, so~$\N_B(\psi\circ g^n(Y)) = X$. By the $R$-closeness of $\psi\circ g^n$ and $f_k^n\circ \psi$, we also get 
\begin{equation}\label{imageF2}
\N_{B+R}(f_k^n\circ\psi(Y)) = X.
\end{equation}
On the other hand, from \eqref{imageF} we obtain that $f_k^n\circ\psi(Y) \subseteq f_k^n(X) = X_{k-n}$, and for $n>k$, we have $X_{k-n} \subseteq X_{-1}$. This is a contradiction with \eqref{imageF2} because $\N_C(f_k^n\circ\psi(Y)) \subseteq \N_C(X_{-1}) = \N_C(\{m^2 \st m\in \NN_{\geq 1}\})$ is a proper subset of $X=[1,\infty)$ for every $C>0$.
\end{proof}

The following remark shows that the analogue of the second bullet point of \cref{thm:squares} with $\psi$ replaced by $\varphi$ does not hold.

\begin{rem} Let $k\in \NN_{\geq 1}$ and let $f_k$ be as in \cref{ex:squares}. There is a coarse equivalence $\varphi_k\colon X\to Y$ such that for all $n\in \NN_{\geq 1}$ the map $\varphi_k\circ f_k^{n}$ equals $g^n\circ \varphi_k$, in particular these compositions are close.
\end{rem}
\begin{proof} We define $\varphi_k =\phi_k$ for $\phi_k$ from formula \eqref{phiK}.
Then, if we put $n=0$ in the calculation from line \eqref{closeness-ceil}, we see that $\varphi_k$ is close to the identity map on $[1,\infty)$, and hence a coarse equivalence.

Recall from the proof of
\cref{thm:squares} that, first, $f_k^{n} = g^n \circ \phi_k$, and, second, the map $g$ preserves $X_k$, $\phi_k(X)  = X_k$, and $\phi_k$ on $X_k$ is the identity. These facts give respectively the two equalities below:
\[\phi_k\circ f_k^{n} = \phi_k \circ g^n \circ \phi_k = g^n \circ \phi_k,\]
proving that indeed $\varphi_k\circ f_k^{n}=g^n\circ \varphi_k$.
\end{proof}

Consider the following statement for two metric spaces $X$ and $Y$ with self-maps $f\colon X\to X$ and $g\colon Y\to Y$: ``There is a coarse equivalence $\varphi\colon X\to Y$ such that $\varphi\circ f$ is close to $g\circ \varphi$''. This statement is precisely a half of the hypothesis of \cref{conjecture}. Example~2.7 of~\cite{GM} exemplifies that this statement is strictly weaker than coarse conjugacy, and in fact even strictly weaker than the hypothesis of \cref{conjecture}. Interestingly, while \cite{GM}*{Example~2.7} is related to \cref{conjecture}, we~prove in \cref{thm:squares} that a similar dynamical system yields counterexamples to \cref{conjecture2}.

In the proof of \cref{thm:squares}, the reason for the lack of coarse conjugacy between $f_k^n$ and $g^n$ for $n>k$ is exactly the fact that $f_k^n$ is not coarsely surjective and $g^n$~is.
(Coarse surjectivity means that the image is coarsely dense in the codomain, and a subset $A$ in a metric space $Z$ is coarsely dense if there exists a~constant $B \geq 0$ such that $\N_B(A)=Z$.) 
While in \cite{GM}*{Example~2.7} neither of $f$ and~$g$ is coarsely surjective and \cite{GM} uses a different argument to show that they are not coarsely conjugate, for certain natural subspaces the restrictions $f'$ of $f$ and $g'$ of~$g$ can be seen to be non-conjugate because $f'$ is not coarsely surjective and $g'$ is. 

\bigskip

The following \cref{ex:invertible} provides counterexamples to \cref{conjecture2} with the additional feature that both compared dynamical systems are invertible. In~particular, in order to show the lack of coarse conjugacy, we need a~more subtle invariant than coarse surjectivity discussed above.

\begin{longEx}\label{ex:invertible}
For every $k\in \NN_{\geq 1}$ let $X_k=Y_k = [0,\infty) \times \set{\range{0}{k}}  \subseteq \RR^2$ and let $f_k\colon X\to X$ and $g_k\colon Y\to Y$ be the bijective maps defined as follows. For $k\in \NN_{\geq 0}$ let
\[g_k(r,j) = (2^{\min(j,1)} r,\, j-1 \bmod k+1).\]
For $k\in \NN_{\geq 1}$, the map $f_k$ is the union of $g_{k-1}$ and multiplication by $2$ on the $k$th copy of the halfline $[0,\infty)$, namely:
\[f_k(r,j) = \begin{cases}
(2^{\min(j,1)} r,\; j-1 \bmod k) & \text{if } j\leq k-1,\\
(2 r,\; j) & \text{if } j = k.
\end{cases}
\]
\end{longEx}

 The second bullet point in \cref{invertible} is weaker than in \cref{thm:squares}: This deficiency is later removed in \cref{asdf} with a more delicate argument.

\begin{thm}\label{invertible}
Let $k\in \NN_{\geq 1}$ and let $X_k=Y_k = [0,\infty) \times \set{\range{0}{k}}  \subseteq \RR^2$ and the bijective maps $f_k\colon X\to X$ and $g_k\colon Y\to Y$ be as in \cref{ex:invertible}. Then
\begin{itemize}
\item for all $n \in \set{\range 1 k}$ the maps $f_k^n$ and $g_k^n$ are coarsely conjugate via the identity maps;
\item but there is no coarse equivalence $\varphi\colon X_k\to Y_k$ such that $\varphi \circ f_k^{k+1}$ and $g_k^{k+1}\circ \varphi$ are close; in particular, $f_k^{k+1}$ and $g_k^{k+1}$ are not coarsely conjugate.
\end{itemize}
\end{thm}

\begin{proof}
For $k\in \NN_{\geq 1}$ and $0\leq n\leq k$ we have
\[g_k^n(r,j) = \begin{cases}
(2^{n-1} r,\;    j-n \bmod k + 1) &\text{if }j < n,\\
(2^n r,\;    j-n \bmod k + 1) &\text{if }n \leq j ,
\end{cases}
\]
and
\[f_k^n(r,j) = \begin{cases}
(2^{n-1} r,\; j-n \bmod k) &\text{if }j < n,\\
(2^n r,\; j-n \bmod k) &\text{if } n \leq j < k , \\
(2^n r,\; j) &\text{if } j=k
\end{cases}
\]
(for $n=k$, the second condition $n \leq j < k$ is empty),
that is, $f_k^n$ differs from $g_k^n$ only at the second coordinate, so these maps are close, and hence coarsely conjugate via the identity maps.

However, for $n=k+1$ we have
\begin{equation}\label{x}
g_k^{k+1}(r,j) = (2^k r,j),
\end{equation}
and
\begin{equation}\label{y}
f_k^{k+1}(r,j) = \begin{cases}
(2^{k-1+\min(j,1)} r,\; j-1 \bmod k) &\text{if }j < k,\\
(2^{k+1} r,\, j) &\text{if } j=k.
\end{cases}
\end{equation}
Since $g_k^{k+1}$ is the identity on the second coordinate and multiplication by~$2^k$ on the first coordinate, it is a controlled map. However, $f_k^{k+1}$ is not controlled: For $m\geq 0$ the points $(m,0)$ and $(m,k)$ are $k$ apart, but they are mapped to the points $(2^{k-1}m, k-1)$ and $(2^{k+1}m, k)$, whose maximum distance $\max(3\cdot 2^{k-1}m, 1)$ goes to infinity with~$m$. Now, if $\varphi\colon X_k\to Y_k$ is a coarse equivalence, $g_k^{k+1}\circ \varphi$ is controlled as a composition of controlled maps, while it is easy to see that $\varphi \circ f_k^{k+1}$ is not controlled (cf.~Proposition~2.8 in \cite{GM}). Hence, $\varphi \circ f_k^{k+1}$ and $g_k^{k+1}\circ \varphi$ cannot be close.
\end{proof}

\begin{cor}\label{moreNonConjugacies}
Let $k\in \NN_{\geq 1}$ and let $X_k$, $Y_k$, $f_k$, and $g_k$ be as in \cref{ex:invertible}.
Then, for every $l\in \NN_{\geq 1}$
there is no coarse equivalence $\varphi\colon X_k\to Y_k$ such that $\varphi \circ f_k^{l(k+1)}$ and $g_k^{l(k+1)}\circ \varphi$ are close; in particular, $f_k^{l(k+1)}$ and $g_k^{l(k+1)}$ are not coarsely conjugate.
\end{cor}
\begin{proof}
The map $g_k^{l(k+1)}$ satisfies $g_k^{l(k+1)}(r,j) = (2^{lk} r,j)$ by \eqref{x} and hence is controlled. On the other hand, $f_k^{l(k+1)}(r,k) = (2^{l(k+1)} r, k)$ by \eqref{y} and
\[f_k^{l(k+1)}(r,0) = \left(2^{l(k+1) - \ceil{l(k+1)/k}} r,\; -l \bmod k\right),\]
so $f_k^{l(k+1)}$ is not controlled for similar reasons as $f_k^{k+1}$ in the proof of  \cref{invertible}. Therefore, the claim follows from the same argument as in the proof of \cref{invertible}.
\end{proof}

The following \cref{qwerty} will be used twice: in the proof of \cref{asdf} below and crucially in \cref{main-section}.
The auxiliary \cref{QIlocal} is a special case of the standard fact that a coarse equivalence between \emph{quasi-geodesic} spaces is a \emph{quasi-isometry}. We
provide a short proof for self-containment.

\begin{lem}\label{QIlocal}
Let $X$ be a metric space, $K\in \NN_{\geq 1}$, and $Y = [0,\infty)\times \{1,\ldots, K\} \subseteq \RR^2$. If $\Phi\colon X\to Y$ is a coarse equivalence, then there are constants $C>0$ and $A\geq 0$ such that for every $x,x' \in X$
\begin{equation*}
d(\Phi(x), \Phi(x')) \geq \tfrac{d(x,x')}{C} - A.
\end{equation*}
\end{lem}

\begin{proof}
We have not specified the norm on $\RR^2\supseteq Y$ because it is irrelevant by the equivalence of all norms, but for the proof let us fix the maximum norm.

By definition, $\Phi$ admits a coarse inverse $\Psi\colon Y\to X$.
Since $\Psi$ is a controlled map, there is a non-decreasing $\rho \colon [0,\infty)\to [0,\infty)$ such that
\[d(\Psi(y), \Psi(y')) \leq \rho\circ d(y,y')\]
for all $y,y'\in Y$.

Let $y,y'\in Y$. For $n=\ceil{d(y,y')}$ there is a sequence $y=y_0, y_1, \ldots, y_n=y'$ in~$Y$ such that $d(y_i, y_{i+1})\leq 1$ for all $i < n$. Indeed, for $(s,k) = y$ and $(s',k') = y'$, one can take $y_i \defeq \big(\frac{(n-i)s + is'}{n}, \big\lfloor \frac{(n-i)k + ik'}{n}\big\rfloor\big)$. (The existence of such sequences is also crucial in 
\cref{lemma-split}.) Hence, by the triangle inequality, we have
\begin{equation*}\label{QI2}
\begin{aligned}
d(\Psi(y), \Psi(y')) &\leq \sum_{i < n} d(\Psi(y_i), \Psi(y_{i+1})) \leq \sum_{i < n} \rho \circ d(y_i, y_{i+1}) \\
&\leq n \cdot \rho(1) \leq \rho(1) \cdot d(y,y') + \rho(1)
\end{aligned}
\end{equation*}
(in particular it follows that $\rho(1)\neq 0$).

Let $B$ be such that $\Psi\circ \Phi$ is $B$-close to the identity on $X$.
We have
\[d(x,x') \leq d(\Psi\circ\Phi(x), \Psi\circ\Phi(x')) + 2B \leq \rho(1)\cdot d(\Phi(x), \Phi(x')) + \rho(1) + 2B,\]
which gives
\begin{equation*}
d(\Phi(x), \Phi(x')) \geq \tfrac{d(x,x')}{\rho(1)} - 1 - \tfrac{2B}{\rho(1)}.\qedhere
\end{equation*}
\end{proof}

\begin{prop}\label{qwerty} Let $J\text{, }K\in \NN_{\geq 1}$ and let $X = [0,\infty)\times \{1,\ldots, J\}$ and $Y = [0,\infty)\times \{1,\ldots, K\}$ be subsets of $\RR^2$. Assume that $f\colon X\to X$, $x_0\in X$, and $F>1$ are such that $r_n\geq F^n$ for $(r_n, j_n) \defeq f^n(x_0)$ and every $n\in \NN_{\geq 1}$. Assume further that $g\colon Y \to Y$ and $G<F$ are such that $t\leq G s$ for $(t, l)\defeq g(s,k)$ and every $(s, k)\in Y$.
Then there is no coarse equivalence $\Phi\colon X\to Y$ such that $\Phi\circ f$ and $g\circ\Phi$ are close; in~particular $f$ and $g$
are not coarsely conjugate.
\end{prop}
\begin{proof}
Suppose for a contradiction that there exist a coarse equivalence $\Phi\colon X\to Y$  and $D\geq 0$ such that $\Phi\circ f$ is $D$-close to $g\circ\Phi$.
Denote $(s, k)\defeq \Phi(x_0)$, and then for $(t_1, l_1)\defeq g\circ \Phi(x_0)$ we have $t_1\leq G s$. If we denote $(s_1,k_1)\defeq \Phi\circ f(x_0)$, then in particular $|t_1-s_1| \leq D$, and hence $s_1\leq G s + D$.

Then for $(t_2,l_2) \defeq g \circ \Phi \circ f(x_0)$ we have
\[t_2 \leq G s_1 \leq G^2 s + G D,\]
and for $(s_2,k_2)\defeq \Phi \circ f \circ f(x_0)$ we get
\[s_2 \leq t_2 + D \leq G^2 s + G D + D.\]
Similarly, for $(t_3, l_3) \defeq g\circ \Phi \circ f^2(x_0)$ we obtain
\[t_3 \leq G s_2 \leq G^3 s + G^2 D + G D,\]
and for $(s_3, k_3)\defeq \Phi \circ f \circ f^2(x_0)$ we have
\[s_3 \leq t_3 + D \leq G^3 s + G^2 D + G D + D.\]
Inductively, we conclude that for $(s_n, k_n) \defeq \Phi \circ f^n(x_0)$ one has
\begin{equation}\label{contradiction-part-1a}
s_n \leq \sum_{i=0}^n G^i \max(D,s) = \frac{G^{n+1} - 1}{G - 1} \max(D,s) = c G^n + a,
\end{equation}
where $c \defeq G\max(D,s)/(G-1)$ and $a \defeq -\max(D,s)/(G-1)$ do not depend on~$n$.

Denote $(s', k') \defeq \Phi(0,0)$. Applying \cref{QIlocal} to the pair of points $f^n(x_0)$ and~$(0,0)$, we obtain
\begin{gather}
\begin{aligned}\label{contradiction-part-2a}
\max\left(|s_n - s'|,\, |k_n - k'|\right) &= d\left(\Phi\left(f^n(x_0)\right), \Phi(0,0)\right) \\
&\geq \frac{d\left(f^n(x_0), (0,0)\right)}{C} - A \\
&\geq \frac{F^n}{C} - A.
\end{aligned}
\end{gather}
Since $s'$ and $k'$ are fixed, and $k_n$ is bounded by $K$, inequality \eqref{contradiction-part-2a} implies that~$s_n$ grows at least as fast as the exponential function with base $F$, which yields a~contradiction with \eqref{contradiction-part-1a} because $F > G$.
\end{proof}

The following strengthens \cref{moreNonConjugacies}, showing that $f_k^n$ and $g_k^n$ are not coarsely conjugate for all $n>k$, not only for $n$ of the form $l(k+1)$.

\begin{thm}\label{asdf} 
Let $k\in \NN_{\geq 1}$ and let $X_k$, $Y_k$, $f_k$, and $g_k$ be as in \cref{ex:invertible}. Then, for every $n > k$ there is no coarse equivalence $\varphi\colon X_k\to Y_k$ such that $\varphi \circ f_k^{n}$ and $g_k^{n}\circ \varphi$ are close; in particular, $f_k^{n}$ and $g_k^{n}$ are not coarsely conjugate.
\end{thm}

\begin{proof}
For $n>k$, we have
\[g_k^{n}(r,j) = \left(2^{n-\ceil{\frac{n-j}{k+1}}} r,\, j-n\bmod k+1\right).\]
Importantly, $2^{n-\ceil{\frac{n-j}{k+1}}} \leq 2^{n-1}$ since $j\leq k < n$. On the other hand, $f_k^n(r,k) = (2^n r,k)$. Hence, the claim follows from \cref{qwerty} for $F=2^n$, $G=2^{n-1}$, and $x_0=(1,k)$.
\end{proof}

\section{A true version of and a counterexample to \texorpdfstring{\cref{conjecture}}{Conjecture \ref{conjecture}}}\label{main-section}

\subsection{Surjectivity of conjugating maps}\label{subsec:theorem}

While in \cref{subsec:main-sec} we provide counterexamples to \cref{conjecture}, we show below
that if one assumes $\varphi$ to be surjective, then \cref{conjecture} holds even after omitting all the conditions on $\psi$ (of course the roles of $\varphi$ and $\psi$ are symmetric).
Hence, the following \cref{surjective-conjecture} is another true version of \cref{conjecture} after \cite{GM}*{Proposition~2.8} (recalled as \cref{thmEmbeddings} in the present note); in fact it is complementary to \cite{GM}*{Proposition~2.8} as explained in the Introduction.

\begin{prop}\label{surjective-conjecture} If there exists a \emph{surjective} coarse equivalence $\varphi\colon X\to Y$ such that
$\varphi\circ f$ and $g\circ\varphi$ are close, then $f$ and $g$ are coarsely conjugate via $\varphi$ and any $\psi\colon Y\to X$ such that $\varphi \circ \psi = \id_Y$.
\end{prop}
\begin{proof}
The requirement that $\varphi \circ \psi = \id_Y$ is equivalent to the fact that $\psi(y) \in \varphi^{-1}(y)$ for every $y\in Y$. By surjectivity, the inverse images $\varphi^{-1}(y)$ are non-empty, and hence such $\psi$ exist. It is a standard exercise that any such $\psi$ is a coarse equivalence, and that $\psi\circ \varphi$ is close to the identity on $X$.

Pick any $\psi$ as above. From the fact that $\varphi \circ f \sim g\circ \varphi$ (where $\sim$ denotes the closeness of maps), we conclude that
\begin{equation}\label{closeness1}
\psi \circ \varphi \circ f \circ \psi \sim \psi \circ g \circ \varphi \circ \psi = \psi \circ g
\end{equation}
because post-composition with a controlled map and pre-composition preserve closeness. Using again invariance under pre-composition, we also obtain
\begin{equation}\label{closeness2}
\psi \circ \varphi \circ f \circ \psi  \sim \id_X \circ f \circ \psi  = f \circ \psi,
\end{equation}
and putting \eqref{closeness1} and \eqref{closeness2} together we conclude that
$\psi \circ g \sim f \circ \psi$.
\end{proof}

Any injective coarse equivalence $\psi\colon Y\to X$ admits a surjective coarse inverse $\varphi\colon X\to Y$ such that $\varphi\circ\psi = \id_Y$. However, the counterexample from \cref{ex:disjointTheorem} involves injective functions $\varphi$ and $\psi$, showing that surjectivity cannot be replaced with injectivity in \cref{surjective-conjecture}. Important here is not the mere existence of a~surjective coarse equivalence $\varphi$, but the fact that $g\circ \varphi$ is close to $\varphi \circ f$, which gives control of $g$ over all of its domain.

\subsection{Counterexample}\label{subsec:main-sec}
The following is a counterexample to \cref{conjecture}.

\begin{longEx}\label{ex:disjointTheorem}
Let
\begin{align*}
X &= \left\{ \left(n^2, r, k\right) \in \RR^3 \st[\Big] n\in \NN_{\geq 1},\ r\in [0,\infty),\ k\in \{1,\ldots, 2n+1\} \right\}\textrm{ and }\\
Y &=  \left\{ \left(n^2, r, k\right) \in \RR^3 \st[\Big] n\in \NN_{\geq 1},\ r\in [0,\infty),\ k\in \{1,\ldots, 2n\} \right\}\textrm{,}
\end{align*}
and let $g\colon Y\to Y$ be the restriction of the bijective map $f\colon X\to X$ given by
\[f\left(n^2, r, k\right) = \left(n^2, kr, k\right).\]
Let $\varphi\colon X\to Y$ and $\psi\colon Y\to X$ be given by
\[\varphi\left(n^2, r, k\right) = \left((n+1)^2, r, k\right) \quad \text{and} \quad \psi\left(n^2, r, k\right) = \left(n^2, r, k\right). \]
\end{longEx}

\begin{thm}\label{disjointTheorem}
Let $X$, $Y$, $f\colon X\to X$, $g\colon Y\to Y$, $\varphi\colon X\to Y$, and $\psi\colon Y\to X$ be as in \cref{ex:disjointTheorem}. Then, the dynamical systems $(X,f)$ and $(Y,g)$ together with the maps $\varphi$ and $\psi$ are counterexamples to \cref{conjecture}, i.e.\ $\varphi$ and $\psi$ are coarse equivalences, and furthermore $\varphi\circ f$ is close to $g\circ\varphi$ and $\psi\circ g$ is close to
$f\circ\psi$, but $f$ and $g$ are not coarsely conjugate.
\end{thm}

The proof of \cref{disjointTheorem} will be divided into steps. First, let us verify that the hypothesis of \cref{conjecture} is satisfied.

\begin{lem}\label{hypothesis}
Let $f\text{, }g\text{, }\varphi\text{, and }\psi$  be as in \cref{ex:disjointTheorem}. Then both $\varphi$ and $\psi$ are coarse equivalences, and furthermore 
$\varphi\circ f$ is close to $g\circ\varphi$ and $\psi\circ g$ is close to $f\circ\psi$.
\end{lem}

\begin{proof}
Define $\Psi \colon Y \to X$ by
\[\Psi(n^2, r, k) = \left(\max(1, n-1)^2,\; r,\; \min\left(k,\,2\max(1, n-1)+1\right)\right).\]
Then, $\Psi\circ \varphi$ equals the identity on $X$ and
\[\varphi \circ \Psi(n^2, r, k) = \left(\max(n, 2)^2,\; r,\; \min\left(k,\, 2\max(1, n-1)+1\right)\right),\]
so $\varphi \circ \Psi$ is close to the identity on $Y$. Since both $\varphi$ and $\Psi$ are clearly controlled maps, this shows that $\varphi$ is a coarse equivalence.

Define $\Phi\colon X\to Y$ by
\[\Phi\left(n^2, r, k\right) = \left(n^2, r, \min(k, 2n)\right).\]
Then, $\Phi\circ\psi$ equals the identity on $Y$, and $\psi\circ\Phi$ is given by the same formula as~$\Phi$, and in particular it is close to the identity on $X$. Since both $\psi$ and $\Phi$ are clearly controlled maps, this shows that $\psi$ is a coarse equivalence.

It is immediate that $\varphi\circ f = g\circ \varphi$, in particular these compositions are close, and similarly $\psi \circ g = f\circ \psi$.
\end{proof}

To obtain \cref{disjointTheorem}, it remains to show that $f$ and $g$ are not coarsely conjugate. The crucial ingredients are \cref{qwerty} and the following technical result.

\begin{lem}\label{lemma-split} Let $X$ and $Y$ be as in \cref{ex:disjointTheorem}. Assume that $\Phi\colon X\to Y$ is a coarse equivalence, and $\Psi$ is its coarse inverse.
There exist cofinite subsets $\NN_X, \NN_Y$ of $\{n^2 \st n\in \NN_{\geq 1}\}$ such that $\Phi$ and $\Psi$ restrict to mutually coarsely inverse coarse equivalences between the sets $(\NN_X \times \RR^2) \cap X$ and $(\NN_Y \times \RR^2) \cap Y$, and moreover there is $B\geq 0$ and a bijection $F\colon \NN_X\to \NN_Y$ such that for every $n^2 \in \NN_X$
\begin{align*}
\N_B\left(\Phi\left(\{n^2\}\times [0,\infty)\times [2n+1]\right)\right) &= \{F(n^2)\} \times [0,\infty)\times\left[2\sqrt{F(n^2)}\right] \\
\N_B\left(\Psi\left(\{F(n^2)\}\times [0,\infty)\times \left[2\sqrt{F(n^2)}\right]\right)\right) &= \{n^2\} \times [0,\infty)\times [2n+1],
\end{align*}
where $[m]$ denotes the set $\set{\range{1}{m}}$.
\end{lem}

That is, for $n\in\NN_{\geq 1}$ sufficiently large, $\Phi$ and $\Psi$ preserve (interchange) the ``thick halflines'' $\{n^2\}\times [0,\infty) \times [2n+1]$ and $\{F(n^2)\} \times [0,\infty) \times [2\sqrt{F(n^2)}]$.
Let us denote $H_n \defeq \{n^2\}\times [0,\infty) \times [2n+1] \subseteq X$ and $H'_m \defeq \{m^2\}\times [0,\infty) \times [2m] \subseteq Y$.

\begin{proof}[Proof of \cref{lemma-split}]
Since modifying the value of $\Phi$
for a single argument does not affect our assumptions,
we can assume that $\Phi(1,0,1)=(1,0,1)$.
For definiteness, let us fix the maximum metric on $X\text{, }Y\subseteq \RR^3$.
 Let $\rho\colon [0,\infty) \to [0,\infty)$ be a non-decreasing function such that $d(\Phi(x), \Phi(x'))\leq \rho \circ d(x,x')$.

Let us observe that there is $M\in \NN_{\geq 1}$ such that 
\begin{equation}\label{first-halfline-preserved}
\Phi(H_1)  \subseteq H'_1\cup H'_2\cup \ldots \cup H'_M.
\end{equation}
Indeed, first note that for every point $x\in H_1$ there is a sequence $x=x_0, x_1, \ldots, x_n=(1,0,1)$ in $X$ such that $d(x_i, x_{i+1})\leq 1$.
For example, if we denote $(1,r,n)\defeq x$, we can take the sequence $(1,r,n), (1,r-1,n),\ldots, (1,r-\floor{r},n), (1,0,n), \ldots, (1,0,1)$.
Consider the
sequence $(y_i)_{i=0}^n$ given by $y_i \defeq \Phi(x_i)$. It satisfies $d(y_i, y_{i+1})\leq \rho(1)$, $y_0=\Phi(x)$ and $y_n=(1,0,1)$. It follows that $\Phi(x) \in H'_1\cup \ldots \cup H'_M$,
where $M$ is the smallest positive integer such that $2M+1 > \rho(1)$ because for $m>M$ the distance between $H'_m$ and its complement is greater than $\rho(1)$.

Similarly, we obtain that if $\Phi(n^2,r,k) = (m^2, s, l)$, then for $m\leq M$ we have 
\begin{equation}\label{containment}
\Phi(H_n) \subseteq H'_1\cup H'_2\cup \ldots \cup H'_M,
\end{equation}
and for $m>M$ we have
\begin{equation}\label{better-containment}
\Phi(H_n) \subseteq H'_m.
\end{equation}

\smallskip We will now obtain a certain inclusion opposite to \eqref{first-halfline-preserved}, namely that there is $C > 0$ such that 
\begin{equation}\label{neighbourhood}
\N_{C}\left(\Phi(H_1)\right) \supseteq H'_1.
\end{equation}
Since $\Phi$ is a coarse equivalence, there is a function $S\colon [0,\infty)\to [0,\infty)$ such that for every $R > 0$
\[d(x,x') \geq S(R) \text{ implies } d(\Phi(x), \Phi(x')) \geq R.\]
(For instance, if $\Psi\circ\Phi$ is $B$-close to the identity on~$X$ and $d(\Psi(y), \Psi(y'))\leq \eta\circ d(y,y')$ for a non-decreasing function $\eta\colon [0,\infty)\to [0,\infty)$, we get that $d(\Phi(x), \Phi(x')) < R$ implies $d(\Psi\circ\Phi(x), \Psi\circ\Phi(x')) \leq \eta(R)$, and in turn $d(x,x')\leq \eta(R) + 2B$, so we can take any $S(R) >\eta(R) + 2B$.)
In particular, $d((1,0,1),\,x)\to \infty$ implies $d(\Phi(1,0,1), \Phi(x)) \to \infty$, and hence the set
\[A = \left\{s \in [0,\infty) \st[\big] \exists x\in H_1 \text{ s.t.\ } \Phi(x)=(m^2,s,l) \right\}\]
is unbounded (as the projection of $\Phi(H_1)$ onto the first coordinate is bounded by inclusion \eqref{first-halfline-preserved}, and consequently the projection onto the third coordinate is bounded as well). Since we have observed that any two points $y,y' \in \Phi(H_1)$ can be ``connected'' with a sequence $y=y_0, y_1, \ldots, y_n=y'$ with $d(y_i, y_{i+1})\leq \rho(1)$, intervals in the complement of~$A$ have length at most~$\rho(1)$, and hence $\N_{\rho(1)}(A) = [0,\infty)$.
We~conclude \eqref{neighbourhood} for $C= \max(\rho(1), M^2-1, 2M-1)$.\smallskip

Observe that for every $m \in \NN_{\geq 1}$, the diameter of $[2m]$ equal to $2m-1$ is bounded by~$m^2$. Consequently, we get 
\[\N_{m^2}(H'_1) \supseteq H'_1\cup \ldots \cup H'_m = \left(\left[1,(m+1)^2\right)\times [0,\infty)\times \NN_{\geq 1}\right) \cap Y.\]
Hence, from inclusion \eqref{neighbourhood} we deduce that for every $m\in \NN_{\geq 1}$ if one takes
$R_m > C+m^2$ and $t_m\defeq S(R_m) + 1$, one gets
\begin{equation}\label{tails}
\Phi\big([t_m,\infty)\times [0,\infty) \times \NN_{\geq 1}\, \cap\, X\big) \subseteq \left[(m+1)^2, \infty\right)\times [0,\infty) \times \NN_{\geq 1}
\end{equation}
because the left-hand side is at least $R_m$ apart from $\Phi(H_1)$, and hence it is disjoint from $\N_{C+m^2}(\Phi(H_1)) \supseteq \N_{m^2}(H'_1)$.

Since $\Phi$ and $\Psi$ are coarse inverses of each other, there exists $B$ such that $\Psi\circ \Phi$ and $\Phi\circ \Psi$ are $B$-close to the respective identity maps on $X$ and $Y$, and in particular $\N_B(\Psi(Y)) = X$ and $\N_B(\Phi(X)) = Y$. Hence, if $m\in \NN_{\geq 1}$ is so large that $2m-1 > B$, then there exists $n\in \NN_{\geq 1}$ such that
\begin{equation}\label{intersection}
\Phi(H_n)\, \cap\, H'_m \neq \emptyset,
\end{equation}
and if also $m>M$, then the non-emptiness \eqref{intersection} of the intersection
implies
the inclusion \eqref{better-containment}. Such $n$ need not be unique, but it follows from formulae~\eqref{containment} and~\eqref{better-containment} that by increasing $m$ we can require all such $n$ to be arbitrarily large.

Now, let $m_0\in \NN_{\geq 1}$ be such that $m_0 > M$, $2m_0 - 1 > B+1$, and for every $m\geq m_0$ every $n$ satisfying \eqref{intersection} also satisfies
\begin{equation}\label{conditions-on-n}
2n-1 > \max\big( S(2B+4),\, \eta(1),\, B  \big).
\end{equation}

Let $m\geq m_0$. Since 
$H'_m$ is contained in $\N_B(\Phi(X)) = Y$ and $2m-1>B$, the set $H'_m$ 
is contained already in the union of the sets $\N_B(\Phi(H_n))$ over $n$ satisfying formula~\eqref{intersection}. In~particular, the interiors $Z_n$ of the sets $\N_{B+1}(\Phi(H_n))$ over such~$n$ form an open covering of $H'_m$.
Since $2n-1 \geq S(2B+4)$ for such $n$,
the sets~$Z_n$ are disjoint. Hence, by the connectedness of a halfline, every $Z_n$ consists of a number of halflines $L_k\defeq \{m^2\} \times [0,\infty) \times \{k\}$. However, if there is more than one $n$ satisfying~\eqref{intersection}, then there must be $k\in [2m-1]$ such that $L_k\subseteq Z_n$ and $L_{k+1}\subseteq Z_{n'}$ for $n\neq n'$. But then, the distance between $Z_n$ and $Z_{n'}$ is $1$, and hence the distance between $\Phi(H_n)$ and $\Phi(H_{n'})$ is at most $2B+3$, contradicting the inequality $2n-1 \geq S(2B+4)$ from~\eqref{conditions-on-n}.

That is, we have just shown that
for every integer $m\geq m_0$ there exists a \emph{unique} $n\in \NN_{\geq 1}$ such that $\N_B(\Phi(H_n)) \supseteq H'_m$, thus we get the equality
\[\N_B(\Phi(H_n)) = H'_m\]
because $m>M$ and hence \eqref{better-containment} holds.
Denote $\NN_Y \defeq \{m^2 \st m\in \NN_{\geq 1},\ m \geq m_0\}$, define a~function~$G$ as the association $m^2\mapsto n^2$, and put $\NN_X = G(\NN_Y)$. The function~$F$ from the statement is then just the inverse of $G$, and the above displayed equality gives the first displayed equality in the conclusion of \cref{lemma-split}.

Since the composition $\Psi \circ \Phi$ is $B$-close to the identity on $X$, we obtain
\begin{equation}\label{Psi-inclusion}
\begin{aligned}
H_{{G(m^2)^{1/2}}} &\subseteq \N_B(\Psi\circ\Phi(H_{G(m^2)^{1/2}}))\\
&\subseteq \N_B(\Psi(H'_m)).
\end{aligned}
\end{equation}
Inequality \eqref{conditions-on-n} guarantees a number of conditions on all $n^2$ in the image of $G$, namely $n^2$ of the form $G(m^2)$ for some $m^2\in \NN_Y$. In particular, $2n-1 > B$, and hence inclusion \eqref{Psi-inclusion} implies that 
\[H_{G(m^2)^{1/2}} \cap \Psi(H'_m) \neq \emptyset,\]
and by the condition $2n-1 > \eta(1)$ the above non-emptiness of the intersection
implies
the inclusion:
$H_{G(m^2)^{1/2}} \supseteq \Psi(H'_{m})$.
Then, using the inequality $2n-1 > B$ again,
we obtain
\begin{align*}
H_{G(m^2)^{1/2}}   &= \N_B(H_{G(m^2)^{1/2}} ) \\
&\supseteq \N_B(\Psi(H'_{m})),
\end{align*}
which together with inclusion \eqref{Psi-inclusion} yields the equality 
\[H_{G(m^2)^{1/2}} = \N_B(\Psi(H'_{m})), \]
which is the second displayed equality in the conclusion of \cref{lemma-split}.

It remains to check that the set $\NN_X$ is cofinite in $\{n^2 \st n\in \NN_{\geq 1}\}$. Note that its complement consists of squares $n^2$ such that $\Phi(H_n) \subseteq [1,m_0^2)\times [0, \infty) \times \NN_{\geq 1}$,
and by formula \eqref{tails} these belong to the bounded interval $[1,t_{m_0})$.
\end{proof}

\begin{lem}\label{increasing-lemma}
Let $f$ and $g$ be as in \cref{ex:disjointTheorem} and $\Phi$ and $\Psi$ be as in \cref{lemma-split}. Assume that $\Phi\circ f$ and $g\circ\Phi$ are close. For the map $F\colon \NN_X\to \NN_Y$ given by \cref{lemma-split}, we have $F(n^2) \geq (n+1)^2$ for all $n^2\in \NN_X$.
\end{lem}

\begin{proof}
Suppose for a contradiction that $F(n^2) \leq n^2$ for some $n^2\in \NN_X$, and denote $m^2=F(n^2)$. By \cref{lemma-split}, $\Phi$ and $\Psi$ restrict to mutually coarsely inverse coarse equivalences between $H_n = \{n^2\}\times [0,\infty) \times [2n+1]$ and $H'_m = \{m^2\}\times [0,\infty)\times [2m]$. By ignoring the first coordinate, the set $H_n$ is isometric to $[0,\infty)\times \set{\range{1}{2n+1}}$ and $H'_m$ to $[0,\infty)\times \set{\range{1}{2m}}$, so we can apply \cref{qwerty} with $F\defeq 2n+1 \geq 2m+1 > 2m \eqdef G$ and $x_0=(n^2,1,2n+1)\in H_n$ to obtain a contradiction.
\end{proof}

\begin{lem}\label{non-decreasing-lemma}
Let $f$ and $g$ be as in \cref{ex:disjointTheorem} and $\Phi$ and $\Psi$ be as in \cref{lemma-split}. Assume that $\Psi\circ g$ and $f\circ\Psi$ are close. For the map $G\colon \NN_Y\to \NN_X$ being the inverse of the map $F$ given by \cref{lemma-split}, we have $G(m^2) \geq m^2$ for every $m^2\in \NN_Y$.
\end{lem}
\begin{proof} The argument is the same as for \cref{increasing-lemma}.
\end{proof}

\begin{proof}[Proof of \cref{disjointTheorem}]
The facts that $\varphi$ and $\psi$ are coarse equivalences, the maps $\varphi\circ f$ and $g\circ\varphi$ are close (in fact equal), and similarly for the maps $\psi\circ g$ and $f\circ\psi$ were proved in \cref{hypothesis}, so it remains to show that $f$ and $g$ are not coarsely conjugate.

Suppose for a contradiction that there exists a coarse equivalence $\Phi\colon X\to Y$ with a coarse inverse $\Psi$ such that $\Phi\circ f$ is close to $g\circ\Phi$ and $\Psi\circ g$ is close to $f\circ\Psi$. Then, combining \cref{increasing-lemma} with \cref{non-decreasing-lemma} gives $G\circ F(n^2) \geq (n+1)^2$, which contradicts the fact that $G$ and $F$ are each other's inverses.
\end{proof}

\section*{Acknowledgements}
I thank William Geller, Dawid Kielak, and Michał Misiurewicz for helpful remarks on earlier versions of the manuscript.

\end{document}